\documentclass[11pt,a4paper,english,reqno]{amsart}

\usepackage{amsmath,amsfonts,amssymb} % for math

\usepackage{amsthm} % theorems

\usepackage{graphicx,caption,subcaption} % figures and subfigures

\usepackage{fullpage} % set margins to be reasonable
\usepackage[colorlinks]{hyperref} % add bookmarks to PDF (and some hyperlinks)
\usepackage[parfill]{parskip} % vertical space between paragraphs rather than tabs

\usepackage{color, soul} % better underlining and also highlighting

\setuldepth{x} % the underlines cross through g's

\allowdisplaybreaks % otherwise the align environment is stupid.

\begingroup
    \makeatletter
    \@for\theoremstyle:=definition,remark,plain\do{%
        \expandafter\g@addto@macro\csname th@\theoremstyle\endcsname{%
            \addtolength\thm@preskip\parskip
            }%
        }
\endgroup

\newtheorem{theorem}{Theorem}[section]
\newtheorem{corollary}[theorem]{Corollary}

\newtheorem{lemma}[theorem]{Lemma}

\theoremstyle{definition}
\newtheorem{defn}{Definition}[section]

\title{Monochromatic Hilbert cubes and arithmetic progressions}

\author{J\'{o}zsef Balogh}
\address{Department of Mathematical Sciences, University of Illinois at Urbana-Champaign, IL, USA, and Moscow Institute of Physics and Technology, 9 Institutskiy per., Dolgoprodny, Moscow Region, 141701, Russian Federation}
\email{jobal@math.uiuc.edu}
\thanks{Research of the first author is partially supported by NSF Grant DMS-1500121 and Arnold O. Beckman Research Award (UIUC) Campus Research Board 18132 and the Langan Scholar Fund (UIUC)}

\author{Mikhail Lavrov}
\address{Department of Mathematics, University of Illinois, 1409 W.\/ Green Street, Urbana IL 61801, USA}
\email{mlavrov@illinois.edu}

\author{George Shakan}
\address{Department of Mathematics, University of Illinois, 1409 W.\/ Green Street, Urbana IL 61801, USA}
\email{shakan2@illinois.edu}

\author{Adam Zsolt Wagner}
\address{Department of Mathematics, University of Illinois, 1409 W.\/ Green Street, Urbana IL 61801, USA}
\email{zawagne2@illinois.edu}

\begin{document}

\maketitle

\begin{abstract}
The Van der Waerden number $W(k,r)$ denotes the smallest $n$ such that whenever $[n]$ is $r$--colored there exists a monochromatic arithmetic progression of length $k$. Similarly, the Hilbert cube number $h(k,r)$ denotes the smallest $n$ such that whenever $[n]$ is $r$--colored there exists a monochromatic affine $k$--cube, that is, a set of the form$$\left\{x_0 + \sum_{b \in B} b : B \subseteq A\right\}$$ for some $|A|=k$ and $x_0 \in \mathbb{Z}$.

We show the following relation between the Hilbert cube number and the Van der Waerden number. Let $k \geq 3$ be an integer. Then for every $\epsilon >0$, there is a $c > 0$ such that $$h(k,4) \ge \min\{W(\lfloor c k^2\rfloor, 2), 2^{k^{2.5-\epsilon}}\}.$$  Thus we improve upon state of the art lower bounds for $h(k,4)$ conditional on $W(k,2)$ being significantly larger than $2^k$. In the other direction, this shows that the if the Hilbert cube number is close its state of the art lower bounds, then $W(k,2)$ is at most doubly exponential in $k$. 

We also show the optimal result that for any Sidon set $A \subset \mathbb{Z}$, one has $$\left|\left\{\sum_{b \in B} b : B \subseteq A\right\}\right| = \Omega( |A|^3) .$$\end{abstract}

\section{Introduction}

A $k$--term arithmetic progression (AP) in the integers is a set of the form $$\{x_0 + d j : 0 \leq j \leq k-1\},$$ where $x_0 , d \in \mathbb{Z}$. Recall the famous Van der Waerden theorem. 

\begin{theorem}[Van der Waerden \cite{Wa}, 1927]  Let $k , r \geq 2$ be integers. Then there exists an $n$ such that in any $r$--coloring of $[n]$, at least one color class contains a $k$--term AP.
\end{theorem}

The smallest such $n$ is said to be the {\em Van der Waerden number}, which we denote by $W(k,r)$. The state of the art bounds on $W(k,r)$ are as follows: Berlekamp~\cite{Be} showed for prime $p$ we have $p\cdot 2^p \leq W(p+1,2)$. This result was recently generalized  by Blankenship, Cummings and Taranchuk~\cite{BlCuTa} who showed the following for $p$ prime \begin{equation}\label{vanderlower}
p^{r-1} 2^p \leq W(p+1,r).
\end{equation} 
 Kozik and Shabanov~\cite{KoSh} proved the general lower bound  $c\cdot r^{k-1}\leq W(k,r)$ for all $k \geq 3$, which is a slight improvement over an application of the Lov\'asz local lemma \cite{Sza}. The best known upper bound  for $W(k,r)$ is the breakthrough result of Gowers~\cite{Go} 
\begin{equation*}\label{upperbound}
 W(k,r) \leq 2^{2^{r^{2^{2^{k+9}}}}}.
\end{equation*}
This upper bound has been further improved in the case when $k=3$ by a series of papers by Graham, Solymosi, Bourgain, Sanders and Bloom (see e.g.~\cite{Bl}) to
\begin{equation*}
W(3,r)\leq 2^{cr(\ln r)^4}.
\end{equation*}
For $r=2$, Graham~\cite{Gr} conjectures  \begin{equation}\label{Graham} W(k,2) < 2^{k^2} ,\end{equation}
and offers \$1000 for a proof or disproof.

Prior to Van der Waerden's study of monochromatic APs, Hilbert studied the same problem for affine cubes. 

\begin{defn}
Given a set $A \subseteq \mathbb Z$, its \emph{restricted sumset} is the set
\[
	\Sigma^* A := \left\{\sum_{b \in B} b : B \subseteq A\right\}.
\]
An affine $k$--cube, or Hilbert cube, is a set of integers that has the form $x_0 + \Sigma^* A$ for some $x_0 \in \mathbb{Z}$ and $A \subset \mathbb Z$ with $|A|=k$. 
\end{defn}
We remark that in the literature, a Hilbert cube typically allows repeated elements in $A$ but we do not. All of the literature we mention below, with the exception of \cite{CoFoSu}, allows repeats. It turns out that in all cases their results can be easily transferred to our situation. 
%When $|A|=k$, there are $2^k$ subsets $B \subseteq A$, but their sums may not necessarily be distinct. By the \emph{volume} of a sumset $\Sigma^* A$, we mean the cardinality of the set $\Sigma^*A$. The volume of an affine $k$--cube is defined similarly.

\begin{theorem}[\label{thm:hilbert}Hilbert \cite{Hi}, 1892]   Let $k , r \geq 2$ be integers. Then there exists an $n$ such that any $r$--coloring of $[n]$, at least one color class contains an affine $k$--cube.
\end{theorem}

We denote the smallest such $n$ by $h(k,r)$. Hilbert's proof yields 
\begin{equation}\label{hilbertbound}
h(k,r) \leq r^{((3 + \sqrt{5})/2)^k}.
\end{equation} 
Since every ${k \choose 2}$--term AP is an affine $k$--cube, we have 
\begin{equation}\label{trivial} 
h(k,r) \leq W\left(\tbinom k2,r\right).
\end{equation}
Thus Van der Waerden's theorem implies Theorem~\ref{thm:hilbert} (but not the bound in~(\ref{hilbertbound})). Szemer\'edi \cite{Sz}, in his seminal paper on the density version of Van der Waerden's theorem, proved that $$h(k,r) = O(r)^{2^k}.$$ Hilbert's and Szemer\'{e}di's results are a massive improvement over combining \eqref{trivial} with the state of the art Van der Waerden bounds in \eqref{upperbound}.  The case where $k=2$ was asymptotically solved by Brown, Chung, Erd\H{o}s and Graham~\cite{BrChErGr}, who showed that 
\begin{equation*}
h(2,r) = (1+o(1))r^2.
\end{equation*} Their lower bound uses difference sets arising from finite projective planes, and their upper bound follows from bounds on Sidon sets. Gunderson and R\"{o}dl \cite{Gu} showed that for $k\geq 3$ we have
\begin{equation*}
r^{(1-o(1))\left(2^k-1\right)/k}\leq h(k,r),
\end{equation*}
where $o(1)\rightarrow 0$ as $r\rightarrow \infty$. Recently Conlon, Fox and Sudakov~\cite{CoFoSu} improved the bound of Erd\H{o}s and Spencer~\cite{erdos89} by showing that there exists an absolute constant $c$ such that
\begin{equation}\label{cofosubound}
r^{ck^2}\leq h(k,r).
\end{equation}
This is currently the best lower bound known for small values of $r$. Their proof heavily relies on an inverse Littlewood--Offord type theorem of Nguyen and Vu~\cite{nguyen11}, which we will also use in our proof of our main result. Note that a significant improvement on~(\ref{cofosubound}) would improve on~(\ref{vanderlower}) because of~(\ref{trivial}). Unfortunately, improving the bounds on Van der Waerden numbers is a notoriously difficult problem. To circumvent this problem, in this paper we focus on improving the Conlon--Fox--Sudakov bound conditional on the fact that $W(k,2)$ is much bigger than $2^k$.

\begin{theorem}\label{main} Let $k \geq 3$ be an integer. Then for every $\epsilon >0$, there is a $c > 0$ such that $$h(k,4) \ge \min\{W(\lfloor c k^2 \rfloor, 2), 2^{k^{2.5-\epsilon}}\}.$$
\end{theorem}

Theorem~\ref{main} asserts that either 

\begin{itemize}
\item[(i)] the lower bound for $W(k,2)$ in \eqref{vanderlower} is far from sharp and $h(k,4)$ is larger than \eqref{cofosubound},
\item[(ii)] the lower bound for $W(k,2)$ in \eqref{vanderlower} is nearly sharp and we can  roughly reverse \eqref{trivial}.
\end{itemize}

We remark that by Theorem~\ref{main}, one can solve Graham's conjecture in \eqref{Graham} by providing an upper bound of $h(k,4) < 2^{k^{2.5-\epsilon}}$. Our proof of Theorem~\ref{main} can be easily adapted to provide lower bounds for $h(k,r)$ where $r>1$ is a square of an integer. We briefly mention that Hilbert cubes have played a central role in upper bounds for van der Waerden numbers \cite{Go, Sz}, via Gower's uniformity norms and Szemer\'edi's cube lemma.

%We briefly mention that affine $k$--cubes also play an important role in Gower's proof \cite{Go} of the density version of Van der Waerden theorem, which provides the current state of the art upper bounds for $W(k,r)$. He iterates the Cauchy--Schwarz inequality to prove the following relation of the so-called Gower's uniformity norm $$\left( \frac{1}{N^2}\sum_{d , n \in \mathbb{Z} / n \mathbb{Z}} f(n) f(n+ d)\cdot \ldots \cdot f(n + (k-1)d) \right)^{2^{k-1}} \leq \frac{1}{N^k}\sum_{x_0 , \ldots , x_{k-1} \in \mathbb{Z} / n \mathbb{Z}} \prod_{B  \subset \{x_1 , \ldots , x_{k-1}\}} f(x_0 + \sum_{b \in B} b ).$$ Taking $f$ to the characteristic function of a set, the above inequality asserts that the average number of arithmetic progressions to the $2^{k-1}$ power is at most the average number of (possibly degenerate) affine $(k-1)$--cubes. 

%In light of Szemer\'edi's cube lemma and Gower's uniformity norms we see that affine $k$--cubes have played a central role in the study of upper bounds for Van der Waerden numbers. 

The idea for the proof of  Theorem~\ref{main} is the following. In a random coloring of $[n]$, the probability that an affine $k$--cube, $x_0 + \Sigma^* A$ is monochromatic is $$\frac{2}{2^{|\Sigma^* A|}}.$$ 
This probability is small when $|\Sigma^* A|$ is large. When $|\Sigma^* A|$ is small, then $A$ should look much like an AP and we are led back to the Van der Waerden problem. Our main tools for making this argument rigorous are a paper of Nguyen and Vu \cite{nguyen11} (see Theorem~\ref{nguyen11}) concerning the Littlewood--Offord theory and another paper of Szemer\'edi and Vu \cite{szemeredi06} (see Theorem~\ref{szemeredi06}) on finding long APs in restricted sumsets, along with some analysis of our own (see Lemma \ref{lemma:dense-gap}) of the case when $A$ is a large subset of a generalized AP.

To prove Theorem~\ref{main}, we analyze which $A \subset \mathbb{Z}$ satisfy 
$$ |\Sigma^* A| = O(|A|^{2.5 - \epsilon}).$$
We conclude this implies $A$ has some additive structure, which eventually yields Theorem~\ref{main}. Curiously, after this analysis we cannot rule out the case that $A$ is a Sidon set, that is $|A+A| = { |A|+1 \choose 2}$. We use different techniques to handle this case.

\begin{theorem}\label{thm:sidon-sumset}
There exists a $c>0$ such that for any Sidon set $A \subset \mathbb{Z}$ one has $$|\Sigma^* A| \geq c |A|^3.$$
\end{theorem}

This result is a side product of our methods and we believe it is of independent interest. The proof is elementary, self-contained and best possible up to the constant $c$. To see this last point, recall the classical result that $[n]$ contains a Sidon set, say $A$, of size $n^{1/2}(1-o(1))$ (see e.g.~\cite[Theorem $5$]{sidonsurvey}). Thus $\Sigma^* A \subset [n^{3/2}]$ and $$|\Sigma^* A| \leq  |A|^3(1+o(1)).$$ 

We briefly mention a related theorem of finding monochromatic Folkman cubes, a wide generalization of Schur's theorem that was obtained independently by Folkman, Rado and Sanders. This generalization is now commonly referred
to as Folkman's theorem (see for example~\cite{GrRoSp}). Let $F(k,r)$ be the smallest $n$ such in that any $r$--coloring of $[n]$ one can find a set $A$ of size $k$ such that $\Sigma^*A\subseteq [n]$ and $\Sigma^*A$ is monochromatic. The state of the art bounds on $F(k,r)$ are significantly different from the best bounds on $H(k,r)$. Indeed, already for $F(k,2)$ the best upper bound due to Taylor~\cite{Ta} is tower-type, while the best lower bound is due to Balogh--Eberhard--Narayanan--Treglown--Wagner~\cite{BaEbNaTrWa}. They are as follows:
\begin{equation*}
2^{2^{k-1}/k} \leq F(k,2)\leq 2^{2^{3^{2^{.^{.^{.^{3}}}}}}},
\end{equation*}
where the tower on the right side has height $4k-3$.

\section{Initial set-up and the random coloring}
Generalized APs play a central role in our argument.

\begin{defn}
A \emph{generalized AP (GAP) of rank $r$} is a set of the form
\[
	Q = \left\{a + \sum_{i=1}^r k_i d_i  : m_i < k_i \le M_i \text{ for $1 \leq i \leq r$}\right\}.
\]

for some  $a, m_1,\dots,m_r$, $M_1,\dots,M_r \in \mathbb{Z}$, and $d_1,\dots,d_r \in \mathbb{Z}$. The \emph{volume} of $Q$ is $(M_1 - m_1) \cdots (M_r - m_r)$. We say $Q$ is \emph{proper} if its volume is equal to its size. We say $Q$ is \emph{symmetric} if $m_i = -M_i$ for $1 \leq i\leq r$. 
\end{defn}

\begin{proof}[Proof idea]
We let $N = \min\{W(\lfloor c k^2 \rfloor, 2), 2^{k^{2.5-\epsilon}/(10\log k)}\}-1$ where $c >0$ is a sufficiently small, fixed constant that depends on our argument. We color $[N]$ by a product coloring $\chi_1 \times \chi_2$, where
\begin{itemize}
	\item $\chi_1 : [N] \to [2]$ avoids monochromatic APs of length $\lfloor c k^2 \rfloor$, 
	\item $\chi_2 : [N] \to [2]$ is a uniformly random coloring.
\end{itemize}

If a Hilbert cube has many distinct elements, then the coloring $\chi_2$ makes sure it is not monochromatic. We will show that all Hilbert cubes having very few distinct elements will contain a $ck^2$--term AP, in which case $\chi_1$ ensures it is not monochromatic.
\end{proof}

To understand $\chi_2$, we will need the following lemma, which appears in a short paper of Erd\H{o}s and Spencer~\cite{erdos89}.

\begin{lemma}\label{lem:erdosspencer}
Let $n,k,u\in\mathbb{N}$ be integers with $u\geq k(k+1)/2$. The number of sets $S\subseteq [n]$ of size $k$ satisfying $|\Sigma^*S| \leq u$ is at most $(kn)^{\log u}u^{2k}$. 
\end{lemma}
The critical case for our purposes is $u = k^a$ for some $a = O(1)$ and $k$ a fixed power of $\log n$ and so the bound in Lemma~\ref{lem:erdosspencer} is $n^{a \log k(1 + o(1))}$.  In this case it is easy to see that $\sim_k n^{a}$ proper GAPs of rank $a-1$ satisfy the hypothesis of Lemma~\ref{lem:erdosspencer}. The additional $\log k$ in the exponent is not concerning for our purposes. A corollary of Lemma~\ref{lem:erdosspencer} is that a random coloring is unlikely to contain Hilbert cubes of large size.

\begin{corollary}\label{random}
Fix an arbitrary $a > 2$. If 
\begin{equation}\label{ineq}
	N\leq 2^{k^{a}/(10a\log k)}
\end{equation}
and $\chi_2: [N]\rightarrow \{0,1\}$ is the uniform random $2$--coloring then w.h.p. (as $k \to \infty$, $a$ fixed) $\chi_2$ does not contain a monochromatic Hilbert cube of size at least $k^{a}$. 
\end{corollary}
\begin{proof}
The probability that a Hilbert cube of size $u$ is monochromatic under $\chi_2$ is $2^{1-u}$. By Lemma~\ref{lem:erdosspencer}  the number of such cubes is $\leq N (kN)^{\log u} u^{2k}$, since we have at most $N$ choices for $x_0$. By the union bound and \eqref{ineq} the probability, $p(k)$, that there is a monochromatic Hilbert cube of size at least $k^{a}$ satisfies
\begin{equation}
p(k) \leq N \sum_{u\geq k^{a}}\left(k 2^{k^{a}/(10a\log k)}\right)^{\log u}u^{2k}2^{1-u}
 \leq kN^2\left(k 2^{k^{a}/(10a\log k)}\right)^{a\log k}k^{2ak}2^{1-k^{a}} =o(1). \qedhere
\end{equation}

\end{proof}

\section{The AP--avoiding coloring}

To analyze a Hilbert cube $x_0 + \Sigma^* A$, we ignore $x_0$ and focus on the structure of $A$. We assume that \begin{equation}\label{size} |\Sigma^* A| \le k^{2.5-\epsilon}, \end{equation} since we Corollary~\ref{random} implies we may choose a $\chi_2$ so that all Hilbert cubes not satisfying \eqref{size} are not monochromatic.

We proceed in several steps. First, we use a result of Nguyen and Vu~\cite{nguyen11} to show that for sets satisfying \eqref{size}, at least half of the set $A$ must be contained in a GAP of small rank and volume. We consider two cases for the rank of the resulting GAP, and show that in each case $x_0 + \Sigma^* A$ is not monochromatic in $\chi_1 \times \chi_2$.

\subsection{Results concerning restricted sumsets of GAPs}

We first recall an inverse theorem of Nguyen and Vu.

\begin{theorem}[Nguyen--Vu, special case of Theorem 2.1 in \cite{nguyen11}]\label{nguyen11}
	Let $C$ be a constant, and let $A$ be a $k$--element set with $|\Sigma^* A| \le k^C$. Then there is a proper symmetric rank $r$ GAP, $Q$, such that $|A \cap Q| \ge \frac12 k$ and $|Q| = O(k^{C - r/2})$, where the constant factor may depend on $C$.
	\label{thm:nguyen-vu}
	
In particular, if $A$ satisfies \eqref{size}, then the two possible cases are
\begin{enumerate}
\item $|Q| = O(k^{2-\epsilon})$ and $Q$ is a rank $1$ GAP (an AP),
\item $|Q| = O(k^{1.5-\epsilon})$ and $Q$ is a rank $2$ GAP.
\end{enumerate}
\end{theorem}
To see the last point of Theorem \ref{nguyen11}, note that if $Q$ were to have rank $3$ or greater, then $|Q|=O(k^{1-\epsilon})$, which contradicts that $Q$ contains at least half of the elements of $A$.

\begin{defn} For an integer $\ell$, let $\ell^*A$ denote the subset of $\Sigma^*A$ consisting of sums of exactly $\ell$ distinct elements $$\ell^* A := \{\sum_{b \in B} b : B \subset A ,\  |B| = \ell\}.$$ \end{defn}

\begin{theorem}[Szemer\'edi--Vu, Theorem 7.1 in \cite{szemeredi06}]\label{szemeredi06}
	For any fixed positive integer $r$ there are positive constants $C$ and $c$ depending on $r$ such that the following holds. For any positive integers $n$ and $\ell$ and any set $A \subseteq [n]$ with $\ell \le \frac12|A|$ and $\ell^r|A| \ge Cn$, the set $\ell^*A$ contains a proper GAP of rank $r'$ and size at least $c \ell^{r'}|A|$, for some integer $r' \le r$.
	\label{thm:szemeredi-vu-general}
\end{theorem}

Our standing assumption \eqref{size} is not compatible with $r' \geq 2$ and $\ell  = \Omega(|A|)$ in Theorem~\ref{thm:szemeredi-vu-general} as long as $n = |A|^{O(1)}$. We formulate this in the following corollary. 

\begin{corollary}\label{APcase}
	For every positive integer $r$, there is a constant $C'$ such that the following holds. Let $P$ be an arbitrary AP with $|P|=n$, and let $A \subseteq P$ which satisfies $C' n \leq |A|^{r+1}$ and $|\Sigma^*(A)| = o(|A|^3)$. Then $\Sigma^* A$ contains an AP of length $\Omega_r(|A|^2)$.
	\label{cor:szemeredi-vu}
\end{corollary}
\begin{proof}
If $P = \{a + k d : 1 \le k \le n\}$, then apply the Freiman homomorphism $x \mapsto \frac{x-a}{d}$ maps $P$ to $[n]$, $A$ to a subset of $[n]$, and preserves the size of both $A$ and $\Sigma^*A$. So we may assume that $P = [n]$ in what follows.
	
	We take $C'$ to be $2^r C$, where $C = C(r)$ is the corresponding constant in Theorem~\ref{thm:szemeredi-vu-general}. Then we have $|A|^{r+1} \ge 2^r Cn$, or $(\frac12|A|)^r |A| \ge Cn$. Applying Theorem~\ref{thm:szemeredi-vu-general} with $\ell = \frac12|A|$, we conclude that $\ell^* A$ contains a proper GAP of rank $r'$ for some $r' \le r$, which has size $c (\frac12|A|)^{r'}|A| = \frac{c}{2^{r'}} |A|^{r'+1}$. In particular, $\Sigma^*A \supset \ell^* A$ contains a GAP of size $\Omega_r( |A|^{r' +1})$. For $r' \geq 2$, this is incompatible with our assumption $|\Sigma^*(A)| = o(|A|^3)$ for sufficiently large $|A|$. So we may assume that $r'=1$, and therefore $\Sigma^* A$ contains an AP of length $\Omega_r(|A|^2)$. 
\end{proof}
	
%In the case $r' \ge 2$, the GAP we get has volume at least $\Omega(|A|^3)$, so in particular $|\Sigma^* A|$ is at least $\Omega(|A|^3)$, which is ruled out by our assumption as long as $|A|$ is sufficiently large.

The following corollary is an immediate consequence of Corollary~\ref{APcase} and our choice of $\chi_1$. 

\begin{corollary}\label{APcolor}
Suppose a set $A$ is of size $k$ and contained in an AP of size $O(k^{\alpha})$ for some $\alpha \geq 1$. Then $x_0 + \Sigma^* A$ is not monochromatic in $\chi_1$. 
\end{corollary}
In case (1) of Theorem~\ref{nguyen11}, we have that $A$ is a subset of an AP of length $O(k^{2- \epsilon})$, and so by Corollary~\ref{APcolor},  $x_0 + \Sigma^* A$ is not monochromatic in the coloring $\chi_1$. Thus $x_0 + \Sigma^* A$ is not monochromatic in the product coloring $\chi_1 \times \chi_2$.

\subsection{Completing the proof of Theorem~\ref{main}}

We are now left to analyze case (2) in Theorem~\ref{nguyen11}. Here at least half of the elements of $A$ is contained in a proper, symmetric GAP, $Q$, of rank 2 and size $O(k^{1.5 - \epsilon})$. In this case, $A$ is basically a dense subset of a two-dimensional integer box (ignoring the technicality that while $Q$ is proper, it may not be that $|\Sigma^* Q| = 2^{|Q|} $). In this case, the size of $\Sigma^*A$ is roughly cubic in $|A|$ as is shown by the following lemma.

\begin{lemma}
	\label{lemma:sparse-gap2}
	There is an absolute constant $C$ such that for $A \subseteq [m] \times [n]$, with $|A| \ge C\sqrt{mn}$,
	\[
		|\Sigma^*A| \ge \Omega \left(\frac{|A|^3}{(\log |A|)^4}\right).
	\]
\end{lemma}

This is best possible up to the logarithm, as is seen by taking $A = [m] \times [n]$. We prove Lemma~\ref{lemma:sparse-gap2} in the following section, but first show how it implies Theorem~\ref{main}.

%Thus there is a rank 2 proper and symmetric GAP, $Q$, containing at least half of  has rank with size $O(k^{1.5-\epsilon})$ (see the remarks after Theorem~\ref{thm:nguyen-vu}). In that case, it has the form $Q = \{ i d_1 + j d_2 : -m \le i \le m, -n \le j \le n\}$, with $|A \cap Q| \ge \frac12|A|$. %We know $Q$ is proper: all elements $id_1 + jd_2$ with different values of $i$ and $j$ are distinct.

Since $Q$ is proper, we may decompose it into the following six disjoint sets:
\begin{itemize}
\item $Q_1 = \{ i d_1 + j d_2 : 1 \le i \le m, 1 \le j \le n\}$,
\item $Q_2 = \{ i d_1 + j d_2 : 1 \le i \le m, -n \le j \le -1\}$,
\item $Q_3 = \{ i d_1 + j d_2 : -m \le i \le -1, 1 \le j \le n\}$,
\item $Q_4 = \{ i d_1 + j d_2 : -m \le i \le -1, -n \le j \le -1\}$,
\item $Q_5 = \{ i d_1 : -m \le i \le m\}$,
\item $Q_6 = \{ j d_2 : -n \le j \le n, j \ne 0\}$.
\end{itemize}
At least one of $|A \cap Q_1|, |A \cap Q_2|, \dots, |A \cap Q_6|$ has size at least $\frac1{12}|A|$. If this happens for $Q_5$ or $Q_6$, then we are in the situation of case (1) of Theorem~\ref{nguyen11}, which we already handled in Corollary~\ref{APcolor}.

%set cannot $A\cap Q_5$ or $A \cap Q_6$ is large, it is a set of size at least $\frac{k}{12}$ contained in either $Q_5$ or $Q_6 \cup \{0\}$: an AP of volume at most $O(k^{1.5-\epsilon})$. As a result, Corollary~\ref{cor:szemeredi-vu} tells us that the corresponding subset $\Sigma^* (A \cap Q_5)$ or $\Sigma^*(A \cap Q_6)$ contains an AP of length $\Omega(k^2)$. Once again, the same must hold for the Hilbert cube $a_0 + \Sigma^* A$.

Without loss of generality, by switching the signs of $d_1$ and $d_2$, we may assume $|A \cap Q_1| \geq \frac{|A|}{12}$. 

Let $\phi : Q_1 \to \mathbb{Z}^2$ via $\phi(i d_1 + j d_2 ) = (i ,j)$. Since $Q_1$ is proper, $\phi$ is injective. It follows, for $k$ sufficiently large, that we have  $$|\phi(A \cap Q_1)| = |A\cap Q_1| = \Omega (k) \geq C \sqrt{mn}.$$ Thus we may apply Lemma~\ref{lemma:sparse-gap2} to find that  
\begin{equation} |\Sigma^* \phi(A \cap Q_1)| = \Omega (k^{3- \epsilon}).\end{equation}

Combing this with \eqref{size}, we have that $$|\Sigma^* (A \cap Q_1)| = |A \cap A_1| < |\Sigma^* \phi(A \cap Q_1)|.$$
One may compare the rest of our argument to \cite[Theorem 3.40]{TV}. It follows that there is a ``collision," that is there exist $1 \leq x_1,y_1 \leq km $ and $1 \leq x_2,y_2 \leq kn$, satisfying

 \[
	x_1 d_1 + x_2 d_2 = y_1 d_1 + y_2 d_2.
\]

%Since $Q$ (and therefore $Q_1$) is proper, the elements of $Q_1$ are in bijection with $[m] \times [n]$, with $id_1 + jd_2 \in Q_1$ corresponding to $(i,j) \in [m] \times [n]$.

%The set $A \cap Q_1$ corresponds by this bijection to a set $A' \subseteq [m] \times [n]$ to which Lemma~\ref{lemma:sparse-gap} applies, since $mn = O(k^{1.5-\epsilon})$ and $|A \cap Q_1| = \Omega(k)$, which exceeds $C \sqrt{mn}$ for $k$ sufficiently large. However, distinct elements of $\Sigma^* A'$ do not necessarily correspond to distinct elements of $\Sigma^*A$. There could be a collision, corresponding to an element of $\Sigma^*A$ which has two representations
%\[
%	x_1 d_1 + x_2 d_2 = y_1 d_1 + y_2 d_2
%\]
%with $x_1,y_1$ necessarily at most $km$ and $x_2,y_2$ at most $kn$.

%If a collision did not occur, then the lower bound on $\Sigma^* A'$ from Lemma~\ref{lemma:sparse-gap} would carry over to a lower bound on $\Sigma^* A$, and we would get
%\[
%	|\Sigma^* A| \ge \Omega\left(\frac{|A|^3}{(\log |A|)^4}\right) = \Omega\left(\frac{k^3}{(\log k)^4}\right).
%\]
%But this contradicts our assumption that $|\Sigma^*A| = O(k^{2.5-\epsilon})$. So a collision must occur.
 
This simplifies to $$|d_1| \cdot |x_1 - y_1| = |d_2| \cdot |x_2 -y_2|.$$ Let $$d := \gcd(d_1,d_2).$$ So $\frac{|d_1|}{d}$ divides $\frac{|d_2|}{d}|x_2-y_2|$ and by Euclid's lemma $\frac{|d_1|}{d}$ divides $|x_2 - y_2|$. Thus $$\frac{|d_1|}{d} \leq kn.$$ Similarly, $\frac{|d_1|}{d} \leq km$.

%so that $\frac{|d_1|}{d}|x_1 -y_1| = \frac{|d_2|}{d}|x_2-y_2|,$ so $\frac{|d_1|}{d}$ divides $\frac{|d_2|}{d}|x_2-y_2|$. But $\frac{|d_1|}{d}$ is relatively prime with $\frac{|d_2|}{d}$, so it must divide $|x_2 - y_2|$, and therefore $\frac{|d_1|}{d} \le |x_2 - y_2| \le kn$. Similarly, $\frac{|d_2|}{d} \le km$. 

Let $R$ be the AP with common difference $d$ given below:
\[
	R = \{-m|d_1|-n|d_2|, \ldots, -d, 0, d, \ldots, m|d_1|+n|d_2|\}.
\]
Then $R$ contains every integer divisible by $d$ between the largest and smallest element of $Q$, so $Q \subseteq R$. Moreover,
\[
	|R| \le 1 + 2 \frac{m|d_1| + n|d_2|}{d} \le 1 + 2m \cdot kn + 2n \cdot km = O(k|Q|) = O(k^{2.5-\epsilon}).
\]
By Corollary~\ref{APcolor}, $x_0 + \Sigma^* (A \cap Q_1)$ is not monochromatic in $\chi_1$ and so neither is $x_0 + \Sigma^* A $. Thus we have handled case (2) of Theorem~\ref{nguyen11} and completed the proof of Theorem~\ref{main}.

%Applying Corollary~\ref{cor:szemeredi-vu} to the set $A \cap Q$ inside the progression $R$, we conclude that $\Sigma^* (A \cap Q)$ contains an AP of length $\Omega(k^2)$. Therefore the Hilbert cube $a_0 + \Sigma^*A$ also contains an AP of length $\Omega(k^2)$, which was what we wanted. This concludes the proof of Theorem~\ref{main}.

\section{Restricted sumsets for dense subsets of high dimensional boxes} 
It remains to prove Lemma~\ref{lemma:sparse-gap2}. We work in arbitrary dimensions, which may be of independent interest. In the following lemma, we are only interested in the case $m=1$, but other values of $m$ are useful as a strengthened induction hypothesis. We recall the $m$--fold sum is $$mB := B + \ldots + B,$$ where there are $m$ summands. 

\begin{lemma}
\label{lemma:dense-gap}
	For all integers $d \ge 1$ there exists an absolute constant $C_d$ such that the following holds. 

Suppose that $A \subseteq [N_1] \times [N_2] \times \ldots \times [N_d]$, with density $\alpha = \frac{|A|}{N_1 N_2 \dotsm N_d}$ satisfies $\frac{\alpha}{(\log \alpha^{-1})^{d-i}}  N_i \ge C_d$ for $2 \le i \le d$, and $m$ is an integer. Then
\[
	|m \Sigma^* A| \ge \Omega\left(\frac{|A|^{d+1} m^d}{\log^{d^2} (\alpha^{-1})}\right).
\]
\end{lemma}
\begin{proof}
We induct on $d$. For $d=1$, we will show the stronger $|m \Sigma^*A| \ge O(|A|^2m)$. To begin with, we have $|\Sigma^* A| \ge \binom{|A|+1}{2}$. Let $A = \{a_1, a_2, \dots, a_k\}$; then an increasing sequence of $\binom{k+1}{2}+1$ elements of $\Sigma^*A$ is given by
\begin{align*}
	0    &< a_1 < a_2 < \ldots < a_k \\
	      &< a_1 + a_k < a_2 + a_k < \ldots < a_{k-1} + a_k \\
	      &< a_1 + a_{k-1} + a_k < a_2 + a_{k-1} + a_k< \ldots < a_{k-2} + a_{k-1} + a_k \\
	      &< \ldots < \\
	      &< a_1 + a_2 + \ldots + a_{k-1} + a_k.
\end{align*}
From the estimate $|X+Y| \ge |X|+|Y|-1$ we have $|mX| \ge m|X|-m+1$, and therefore as long as $|A|\geq 2$ we have $|m \Sigma^* A| \ge m\binom{|A|+1}{2} - m + 1 \ge \frac12 |A|^2m$.

For the induction step, assume that this lemma holds in dimension $d-1$, where $d\geq 2$. Partition $A$ into \emph{stacks}
\[
	A_x = \{ a \in A : (a_1, \ldots, a_{d-1}) = x\}
\]
indexed by $x \in [N_1] \times \dots \times [N_{d-1}]$. The average size of a stack $A_x$ is $\alpha N_d$. Call a stack $A_x$ \emph{sparse} if $|A_x| \le \frac12 \alpha N_d$, and \emph{dense} otherwise. Then the total number of elements of $A$ contained in sparse stacks is at most
\[
	\frac12 \alpha N_d \cdot \prod_{i=1}^{d-1} N_i = \frac12|A|,
\]
so at least $\frac12|A|$ elements of $A$ are in dense stacks.

The sizes of dense stacks range from $\frac12\alpha N_d$ to $N_d$. For each $t$ such that $\frac12\alpha N_d \le t \le \frac12 N_d$, define
\[
	X_t = \{x \in [N_1] \times \dots \times [N_{d-1}] : t < |A_x| \le 2t\},
\]
so that $X_t$ indices all stacks whose sizes range from $t$ to $2t$. 
By a dyadic decomposition, we can find a $t$ so that the union of the stacks indexed by $X_t$ is large. That is, letting $s = \left\lceil \log_2 \alpha^{-1} \right\rceil$, we can partition the indices of all the dense stacks into the disjoint union of $s$ sets
\[
	\bigcup_{i=0}^{s-1} X_{2^{i-1} \alpha N_d}.
\]
Since at least $\frac12|A|$ elements of $A$ are in dense stacks, there must be a $t = 2^{i-1} \alpha N_d$ for some $i$ between $0$ and $s-1$ such that at least $\frac1{2s}|A|$ elements of $A$ are in stacks indexed by some $x \in X_t$.

For each $x \in X_t$, we choose two disjoint sets $B_x, C_x \subseteq A_x$, where $|B_x| = 2\lfloor \frac t3\rfloor$ and $|C_x| = \lfloor \frac t3\rfloor$. Since $\alpha N_d \ge C_d$, we have $\lfloor \frac t3 \rfloor \ge \frac t4$, provided that we choose $C_d$ sufficiently large. Define
\[
	C = \bigcup_{x \in X_t} C_x.
\]
We will show that $|m \Sigma^* A|$ is large in two steps.

Let $b \in m\Sigma^*A$ be given by summing the $\lfloor \frac t3\rfloor$ smallest elements of each $B_x$, each with multiplicity $m$. Then $b + m \Sigma^*C$ is a subset of $m \Sigma^* A$. We show that not only is $|b + m\Sigma^* C|$ large, but that its projection onto the first $d-1$ coordinates is large. 

In this projection, the exact elements of each $C_x$ are irrelevant, since their first $d-1$ coordinates are just $x$. Being able to choose the elements of $C_x$, of which there are at least $\frac t4$, up to $m$ times each is equivalent to being able to include $x$ in a sum up to $\frac{mt}{4}$ times, and so the size of the projection is $|\frac{mt}{4} \Sigma^* X_t|$.

Since each stack $A_x$ has size at most $2t$, and the union of all stacks indexed by $X_t$ has size at least $\frac{|A|}{2s}$, we know that $X_t$ itself must have size at least $\frac{|A|}{4st}$. We apply the induction hypothesis to $X_t$. The density of $X_t$ in $[N_1] \times \dots \times [N_{d-1}]$ is at least
\[
	\alpha' = \frac{|A|}{N_1 \dotsm N_{d-1} \cdot 4st} \ge \frac{|A|}{N_1 \dotsm N_d \cdot 4s} \ge \frac{\alpha}{4 \lfloor \log_2 (\alpha^{-1})\rfloor},
\]
which will satisfy the conditions in the induction hypothesis provided we choose $C_d$ sufficiently large compared to $C_{d-1}$. Additionally, $\log \left(\alpha'^{-1}\right) = \Theta(\log \alpha^{-1})$. So we have 
\[
	\left|\frac{mt}{4} \Sigma^* X_t\right| \ge \Omega\left(\frac{|X_t|^d(\frac14 mt)^{d-1}}{(\log\alpha'^{-1})^{(d-1)^2}}\right) \ge \Omega \left(\frac{|A|^{d-1} m^{d-1}t^{-1}}{(\log\alpha^{-1})^{(d-1)^2+d}}\right).
\]
Second, for each element of $b + m \Sigma^*C$, we show that there are many elements of $m \Sigma^* A$ with the same projection onto the first $d-1$ coordinates. We can obtain such elements by replacing $b$ with a different sum which also uses $m\lfloor \frac t3\rfloor$ elements of $B_x$ for each $x$, counting multiplicity.

Let $k = \lfloor \frac t3\rfloor$, and let $B_x = \{b_{x,1}, \dots, b_{x,2k}\}$. For a fixed $x$, there are at least $mk^2$ distinct sums of elements of $B_x$ with total multiplicity $mk$. The argument here is similar to the $d=1$ case of this lemma. Start with the smallest possible sum,
\[
	\sum_{i=1}^k m b_{x,i}.
\]
For each sum with total multiplicity $mk$, we may increase its $d$\textsuperscript{th} coordinate by choosing the largest $i<2k$ such that $b_{x,i}$ is included in the sum, while $b_{x,i+1}$ is included fewer than $m$ times, and replace $b_{x,i}$ by $b_{x,i+1}$. This ends only when we reach the largest possible sum,
\[
	\sum_{i=k+1}^{2k} m b_{x,i}.
\]
The sum of the indices on the $mb_{x,i}$, taken with multiplicity, starts at $\sum_{i=1}^k mi = m\binom{k+1}{2}$ and ends at $\sum_{i=k+1}^{2k} mi = mk^2 + m \binom{k+1}{2}$. In each step, since we replace some $b_{x,i}$ by $b_{x,i+1}$, the sum of indices increase by $1$, so we take a total of $mk^2 \ge m(\frac t4)^2 = \Omega(mt^2)$ steps.

Now we must aggregate this result over all $x \in X_t$. Starting at the element $b$ as previously defined, go through the elements of $X_t$ arbitrarily, and for each $x \in X_t$, perform the above process, taking $\Omega(mt^2)$ steps. There are a total of $|X_t| \cdot \Omega(mt^2)$ steps taken, and each one increases the $d$\textsuperscript{th} coordinate while leaving the first $d-1$ coordinates unchanged. Altogether, for every $a \in b + m\Sigma^* C$, we obtain
\[
	|X_t| \cdot \Omega(mt^2) = \Omega\left(mt^2 \cdot \frac{|A|}{4st}\right) = \Omega\left(\frac{|A|mt}{\log \alpha^{-1}}\right)
\]
different elements of $m\Sigma^*A$ with the same first $d-1$ coordinates as $a$.

Repeating this for each of the $\Omega\left(\frac{|A|^d m^{d-1} t^{-1}}{\log^{10d-10+d}|A|}\right)$ elements of $b + m\Sigma^*C$, we get
\[
	|m\Sigma^* A| \ge \Omega\left(\frac{|A|^d m^{d-1}t^{-1}}{(\log\alpha^{-1})^{(d-1)^2+d}}\right) \cdot \Omega\left(\frac{|A|mt}{\log \alpha^{-1}}\right) = \Omega\left(\frac{|A|^{d+1}m^d}{(\log \alpha^{-1})^{d^2}}\right),
\]
completing the inductive step.
\end{proof}

In our application we will have $d=2$. In this case, assuming that $A$ is sufficiently large, we get the bound in a second lemma, given below. By applying a Freiman isomorphism, that bound also applies when $d>2$ but the set $A$ is too sparse to use Lemma~\ref{lemma:dense-gap} directly.

\begin{lemma}
	\label{lemma:sparse-gap}
	There is an absolute constant $C$ such that for all $d\ge 2$, the following holds.
	Suppose that $A \subseteq [N_1] \times [N_2] \times \dots \times [N_d]$, with $|A| \ge C\sqrt{N_1 N_2 \cdot \ldots\cdot N_d}$. Then
	\[
		|\Sigma^*A| \ge \Omega \left(\frac{|A|^3}{(\log |A|)^4}\right).
	\]
\end{lemma}
\begin{proof}
First we handle the case $d=2$. Then this result is a direct application of Lemma~\ref{lemma:dense-gap}, once we assure ourselves that it applies. Without loss of generality, assume that $N_1 \le N_2$. Since $|A| \ge C\sqrt{N_1 N_2}$, the density $\alpha = \frac{|A|}{N_1N_2}$ satisfies $\alpha N_2 \ge C \sqrt{\frac{|N_2|}{|N_1|}} \ge C$, so the conditions of Lemma~\ref{lemma:dense-gap} are satisfied for any $C$ which is at least the constant $C_2$ from that lemma.

Taking $m=1$, we conclude that
\[
	|\Sigma^*A| \ge \Omega \left(\frac{|A|^3}{(\log \alpha^{-1})^4}\right) = \Omega \left(\frac{|A|^3}{(\log |A|)^4}\right).
\]
Next, we assume that  $d>2$ and we begin by omitting any coordinates $i$ with $N_i = 1$, so that we may assume $N_i \ge 2$ for all $i$. Therefore if we define
\[
	M = \sum_{i=2}^d N_2 N_3 \cdot \ldots\cdot  N_i,
\] 
we have $M \le (N_2 \cdot \ldots\cdot N_i)(1 + \frac12 + \frac14 + \dotsb) \le 2 N_2 \cdot \ldots\cdot N_i$.

We map $[N_1] \times \dots \times [N_d]$ to $[N_1] \times [M]$ by the homomorphism $\phi\colon \mathbb Z^d \to \mathbb Z^2$ which takes
\[
	(x_1, x_2, \dots, x_d) \mapsto \left(x_1, \sum_{i=2}^d x_i \prod_{j=2}^{i-1} N_j \right).
\]
The homomorphism $\phi$ is injective on $[N_1] \times \dots \times [N_d]$, so the image $\phi(A)$ has the same size as $A$. Therefore
\[
	|\phi(A)| \ge C \sqrt{N_1 N_2 \cdot \ldots\cdot N_d} \ge C \sqrt{\frac{N_1M}{2}},
\]
which is large enough for the $d=2$ case of this lemma to apply if we choose $C = C_2 \sqrt 2$. Applying the $d=2$ case of this lemma,
\[
	|\Sigma^*A| \ge |\Sigma^* \phi(A)| \ge \Omega \left(\frac{|\phi(A)|^3}{(\log |\phi(A)|)^4}\right) = \Omega \left(\frac{|A|^3}{(\log |A|)^4}\right),
\]
which was what we wanted. 
\end{proof}

Lemma~\ref{lemma:sparse-gap2} follows from Lemma~\ref{lemma:sparse-gap} by taking $d =2$.

%\subsection{Completing the proof of Theorem~\ref{main}}
%Applying Theorem~\ref{thm:nguyen-vu} with $C=2.5-\epsilon$ to the set $A$ gives us a GAP $Q$ such that $|A \cap Q| \ge \frac 12|A|$. There are several possible cases for the rank and volume of $Q$.

%In the case that Q has rank 1, it has volume $O(k^{2-\epsilon})$. Inside it, the subset $A \cap Q$ has size at least $\frac12 k$, which is $\Omega(|Q|^{1/2})$, so $|A \cap Q|^2 = \omega(|Q|)$. Applying Corollary~\ref{cor:szemeredi-vu} to set $A\cap Q$ inside the progression $Q$ tells us that $\Sigma^*(A \cap Q)$ contains a proper AP of size $\Omega(|A|^2) = \Omega(k^2)$, and therefore the Hilbert cube $a_0 + \Sigma^*A$ also contains an AP of length $\Omega(k^2)$.
\section{Sidon sets}
%\begin{theorem}
%\label{thm:sidon-sumset}
%	Let $A$ be a Sidon set of positive integers. Then $|\Sigma^*A| = \Omega(|A|^3)$.
%\end{theorem}

We now set out to prove Theorem~\ref{thm:sidon-sumset}. We prove this by starting with a small subset $X \subseteq A$, and adding elements to $X$ slowly while ensuring that $|\Sigma^*X|$ grows quickly. In the end, $|\Sigma^*X|$ will reach $\Omega(|A|^3)$ in size before the set $A$ is exhausted.

As long as $|\Sigma^*X|$ is relatively small, the following lemma guarantees that we can increase $|\Sigma^*X|$ by a factor of $\frac32$ with the addition of only two new elements.

\begin{lemma}
\label{lemma:small-x}
	Let $A$ be a Sidon subset of the positive integers, and let $X \subseteq A$ with $|X| \le \frac12 |A|$ and $|\Sigma^*X| \le \binom{\frac12|A|}{2}$. Then we can extend $X$ to $X' = X \cup \{a_1, a_2\}$ with $a_1, a_2 \in A \setminus X$ in such a way that $|\Sigma^* X'| \ge \frac32 |\Sigma^* X|.$
\end{lemma}
\begin{proof}
	Let $B = \{a_1 + a_2 : a_1, a_2 \in A \setminus X\}$. Since $A$ is Sidon, all elements of $B$ are distinct, so $|B| = \binom{|A \setminus X|}{2} \ge \binom{\frac12 |A|}{2}$. In particular, $|B| \ge |\Sigma^* X|$.

	The total number of solutions of the equation $s_1 + b = s_2$ with $s_1, s_2 \in \Sigma^*X$ and $b \in B$ is at most $\binom{|\Sigma^*X|}{2}$: once we choose the set $\{s_1, s_2\}$, we are forced to choose the order $s_1 < s_2$, and then $b$, if it exists, is unique. So there exists an element $b \in B$ for which there is at most the average number
	\[
		\frac{\binom{|\Sigma^*X|}{2}}{|B|} \le \frac{|\Sigma^*X|^2}{2|B|} \le \frac{|\Sigma^*X|}{2|B|} \cdot |\Sigma^*X| \le \frac12 |\Sigma^*X|
	\]
of solutions. In other words, $|\Sigma^*X \cap (\Sigma^*X + b)| \le \frac12|\Sigma^*X|$.

Write this $b$ as $a_1 + a_2$, and let $X' = X \cup \{a_1, a_2\}$. Then
\[
	|\Sigma^* X'| \ge |\Sigma^*X + (\Sigma^*X + b)| \ge |\Sigma^*X| + |\Sigma^*X + b| - |\Sigma^*X \cap (\Sigma^*X + b)| \ge \frac32 |\Sigma^*X|,
\]
as desired.
\end{proof}
When $|\Sigma^*X|$ is large, the previous lemma does not apply, and we need a second iterative way to increase $|\Sigma^*X|$.

\begin{lemma}
\label{lemma:large-x}
	Let $A$ be a Sidon subset of the positive integers, and let $X \subseteq A$ with $|X| \le \frac34 |A|$ but $|\Sigma^*X| \ge \binom{\frac14|A|}{2}$. Then we can extend $X$ to $X' = X \cup \{a_1, a_2\}$ with $a_1, a_2 \in A \setminus X$ in such a way that $|\Sigma^* X'| \ge |\Sigma^* X| + \frac12\binom{\frac14|A|}{2}.$
\end{lemma}
\begin{proof}
Let $A'$ be a subset of $A$ with $|A'| = \frac14|A|$ and $A' \cap X = \varnothing$, and let $B = \{a_1 + a_2 : a_1, a_2 \in A'\}$. Since $A$ is Sidon, all elements of $B$ are distinct, so $|B| = \binom{|A'|}{2} = \binom{\frac14 |A|}{2}$; in particular, $|B| \le |\Sigma^*X|$.

Let $S$ consist of the $|B|$ largest elements of $|\Sigma^*X|$. Of the $|B|^2$ elements of $S\times B$, at most $\binom{|B|}{2}$ are ordered pairs $(s,b)$ with $s+b \in \Sigma^*X$, because then $s+b$ would be a larger element of $S$, and there are $\binom{|S|}{2} = \binom{|B|}{2}$ pairs of elements of $S$. 

So there are at least $|B|^2 - \binom{|B|}{2} > \frac12|B|^2$ elements of $S \times B$  which are ordered pairs $(s,b)$ with $s+b \notin \Sigma^*X$. By averaging, there is some $b \in B$ contained in at least $\frac12|B|$ of those ordered pairs. For this choice of $b$, $\Sigma^*X + b$ contains at least $\frac12|B|$ values not found in $\Sigma^*X$.

Write $b = a_1 + a_2$ for some $a_1, a_2 \in A'$, and let $X' = X \cup \{a_1, a_2\}$. Then
\[
	|\Sigma^* X'| \ge |\Sigma^*X + (\Sigma^*X + b)| \ge |\Sigma^*X| + \frac12|B| \ge |\Sigma^*X| + \frac12\binom{\frac14|A|}{2},
\]
as desired.
\end{proof}

Now we put together the details and prove Theorem~\ref{thm:sidon-sumset}. 

\begin{proof}
Begin with $X = \varnothing$, and $|\Sigma^* X| = 1$, and repeatedly apply  Lemma~\ref{lemma:small-x} until one of the hypotheses is violated: either $|X| \ge \frac12|A|$ or $|\Sigma^*X| \ge \binom{\frac12|A|}{2}$. In fact, after $k$ iterations of Lemma~\ref{lemma:small-x}, we will have $|\Sigma^*X| \ge (\frac32)^k$, so the second hypothesis will be violated when $|X|$ is only $O(\log |A|)$; for $|A|$ sufficiently large, this will happen first. 

Next, apply Lemma~\ref{lemma:large-x} to $X$ repeatedly, increasing $|X|$ by $2$ at every step while increasing $|\Sigma^*X|$ by $\frac12\binom{\frac14|A|}{2} = \Omega(|A|^2)$. 

It will take more than $\frac18|A|$ applications of Lemma~\ref{lemma:large-x} before the hypothesis that $|X| \le \frac34|A|$ is no longer satisfied. At that point, $|\Sigma^*X|$ will have size at least $\frac18|A| \cdot \frac12\binom{\frac14|A|}{2} = \Omega(|A|^3)$.

In particular, this means that $|\Sigma^*A| = \Omega(|A|^3).$
\end{proof}

%\bibliographystyle{plain}
%\bibliography{C:/Users/Misha/Dropbox/Bibliography/all-papers}

\end{document}